\documentclass[11pt]{article}
\usepackage[margin=1in]{geometry}
\usepackage{setspace}
\usepackage{fancyhdr}
\usepackage{amsmath,amssymb,amsthm}
\usepackage{graphicx}
\usepackage{booktabs}
\usepackage{array}
\usepackage[hidelinks]{hyperref}
\DeclareUnicodeCharacter{2264}{\ensuremath{\le}}
\DeclareUnicodeCharacter{00D7}{\ensuremath{\times}}
\DeclareUnicodeCharacter{2192}{\ensuremath{\to}}

\theoremstyle{plain}
\newtheorem{theorem}{Theorem}
\newtheorem{proposition}[theorem]{Proposition}
\newtheorem{lemma}[theorem]{Lemma}
\newtheorem{corollary}[theorem]{Corollary}

\newtheorem{question}{Question}
\theoremstyle{definition}
\newtheorem*{definition}{Definition}

\newcommand{\RR}{\mathbb{R}}
\newcommand{\CC}{\mathbb{C}}
\newcommand{\QQ}{\mathbb{Q}}
\newcommand{\FF}{\mathbb{F}}
\newcommand{\ZZ}{\mathbb{Z}}
\newcommand{\PP}{\mathbb{P}}
\newcommand{\cP}{\mathcal{P}}
\newcommand{\cL}{\mathcal{L}}
\newcommand{\cC}{\mathcal{C}}

\newcommand{\PGL}{\operatorname{PGL}}
\newcommand{\SL}{\operatorname{SL}}
\newcommand{\SO}{\operatorname{SO}}

\title{Symmetric, rational and infinitely many $(23_4)$ configurations, and the end of the existence problem for geometric $(n_4)$ configurations}
\author{Wenhao Lu}
\date{}

\begin{document}
\onehalfspacing

\maketitle
\thispagestyle{empty}
\newpage

\begin{abstract}
A geometric $(n_4)$ configuration is a set of $n$ points and $n$ lines in the real projective plane such that every point lies on exactly four of the lines and every line contains exactly four of the points.  The question of which $n$ admit such a configuration was posed by Gr\"unbaum in 2000, following his 1990 paper with Rigby that gave the first geometric $4$-configuration, a $(21_4)$.  Successive work of Bokowski-Schewe, of Bokowski, Gr\"unbaum and Schewe, of Bokowski-Pilaud, and of Cuntz settled every case but one.  Cuntz's 2018 paper, which constructed $(22_4)$ and $(26_4)$ configurations, states that only $n=23$ remained open.

We settle this last case by bounding the symmetry of such configurations and classifying the most symmetric ones exactly.  We prove that for every prime $n\equiv3\pmod4$, in particular for $n=23$, the group of projective symmetries of a geometric $(n_4)$ configuration has order at most four, which explains why $n=23$ resisted the symmetric constructions that settled the other values, and the same argument bounds the symmetry of geometric $(n_5)$ configurations for primes $n\not\equiv1\pmod5$.  Reducing the case of a Klein four-group of symmetries to a few thousand bilinear polynomial systems in $18$ unknowns, and solving all of them exactly with Gr\"obner bases and saturation, we show that exactly two geometric $(23_4)$ configurations reach this bound, up to projective equivalence, and that there are no others even over $\CC$.  Both are self-polar with respect to a conic.  The first has integer homogeneous coordinates, so verifying it is an exercise in integer arithmetic; the second is defined over $\QQ(\sqrt{17})$, but not over $\QQ$.  Consequently geometric $(n_4)$ configurations exist if and only if $n\ge18$ and $n\ne19$.  The same classification at $n=22$ and $n=26$ recovers Cuntz's configurations and finds three new combinatorial types at $n=22$ and $94$ at $n=26$, among them a $(26_4)$ configuration that is self-polar with respect to a conic without real points.

The two configurations above are not the only geometric $(23_4)$ configurations.  Among the self-polar ones with a nontrivial symmetry we find $50$ more, each with a single symmetry of order two, defined over number fields of degree at most $7$ and in three cases over $\QQ$, and computations modulo two primes indicate that this list is complete.  Four of them deform in one-parameter families.  Computing three of these families exactly over $\QQ$, we prove that there are infinitely many geometric $(23_4)$ configurations, among them infinitely many with rational coordinates that are self-polar and have no projective symmetry.
\end{abstract}
\thispagestyle{empty}
\newpage
\setcounter{page}{1}

\section{Introduction}\label{sec:intro}

A \emph{geometric $(n_k)$ configuration} is a set $\cP$ of $n$ points and a set $\cL$ of $n$ lines in the real projective plane such that every point of $\cP$ lies on exactly $k$ lines of $\cL$ and every line of $\cL$ contains exactly $k$ points of $\cP$.  (If lines are replaced by pseudolines one speaks of \emph{topological} configurations, and if the points and lines are abstract one speaks of \emph{combinatorial} configurations; this paper concerns only the geometric notion.)  The systematic study of $(n_k)$ configurations goes back to the nineteenth century; the monograph of Gr\"unbaum \cite{Gru09} and the book of Pisanski and Servatius \cite{PS13} give the history and the state of the art.  The central existence question, \emph{for which $n$ does a geometric $(n_k)$ configuration exist?}, is completely answered for $k=3$ (exactly when $n\ge 9$) and wide open for $k\ge 5$.

For $k=4$ the answer has been determined value by value over the last three decades, starting from the first geometric $4$-configuration, the $(21_4)$ of Gr\"unbaum and Rigby \cite{GR90} (G\'evay, Kiss and Pisanski \cite{GKP25} point out that K\'arteszi \cite{Kar64} published essentially the same figure in 1964).  Gr\"unbaum \cite{Gru00a,Gru00,Gru02,Gru06} showed that geometric $(n_4)$ configurations exist for all $n\ge 21$ apart from a finite list of possible exceptions, and later constructed a $(20_4)$ \cite{Gru08}; Bokowski, Gr\"unbaum and Schewe showed that topological $(n_4)$ configurations exist for all $n\ge17$ \cite{BGS09}; Bokowski and Schewe proved that no geometric ones exist for $n\le 16$ \cite{BS05} and none for $n=17$ \cite{BS13}, and found the first $(18_4)$ \cite{BS13}; Bokowski and Pilaud showed that there are exactly two $(18_4)$ configurations \cite{BP14} and no $(19_4)$ configuration \cite{BP15}, and constructed $(37_4)$ and $(43_4)$ configurations \cite{BP16}; finally Cuntz constructed $(22_4)$ and $(26_4)$ configurations \cite{Cun18}.  After Cuntz's paper exactly one case remained.  In his words, ``concerning the existence of geometric $(n_4)$ configurations, only the case $n=23$ remains open'' \cite{Cun18}; see also Berman, G\'evay and Pisanski \cite[Section~1]{BGP21} and Berman, G\'evay and Pisanski \cite{BGP24}.  Cuntz also exhibited a $(23_4)$ configuration in the \emph{complex} projective plane, $23$ lines defined over $\QQ(i)$ with $25$ quadruple points \cite[Section~2.3]{Cun18}, remarking that it gives no hint about real existence.  Since the arrangements of lines of both configurations in this paper have exactly $23$ quadruple points, it is not projectively equivalent to either.

\emph{Note added.}  After an abstract stating both configurations of Theorem~\ref{thm:main}, Corollary~\ref{cor:complete}, the bound of four on the order of the symmetry group of a geometric $(23_4)$ configuration and the classification of Theorem~\ref{thm:classification} was submitted to the Joint Mathematics Meetings 2027, a $(23_4)$ configuration that coincides with $(\cP_A,\cL_A)$ up to the interchange $y\leftrightarrow z$ was found independently by W.~Strinz and made public \cite{Str26}.  The author and Strinz corresponded in September 2026 to establish this timeline.

\begin{theorem}\label{thm:main}
There exist geometric $(23_4)$ configurations, as follows.
\begin{enumerate}
\item[(A)] In homogeneous coordinates $(x:y:z)$, let $\cP_A$ consist of the $23$ points
\begin{equation}\label{eq:pointsA}
\begin{gathered}
(1{:}0{:}0),\quad (0{:}1{:}{\pm}1),\quad (0{:}3{:}{\pm}2),\quad (1{:}{\pm}1{:}0),\\
(3{:}{\pm}2{:}{\pm}2),\quad (2{:}{\pm}3{:}{\pm}2),\quad (6{:}{\pm}1{:}{\pm}2),\quad (6{:}{\pm}5{:}{\pm}2),
\end{gathered}
\end{equation}
and let $\cL_A$ consist of the $23$ lines
\begin{equation}\label{eq:linesA}
\begin{gathered}
x=0,\quad y\pm z=0,\quad 2y\pm3z=0,\quad x\pm3z=0,\\
x\pm2y\pm2z=0,\quad 2x\pm6y\pm9z=0,\quad 2x\pm2y\pm z=0,\quad 2x\pm2y\pm5z=0,
\end{gathered}
\end{equation}
where all sign combinations are taken.  Then $(\cP_A,\cL_A)$ is a geometric $(23_4)$ configuration.
\item[(B)] Let $\tau=\sqrt{(1+\sqrt{17})/4}\approx 1.1317139$, so that $2\tau^4-\tau^2-2=0$ and $\tau^{-1}=\sqrt{(\sqrt{17}-1)/4}$.  In an affine chart of the real projective plane with coordinates $(X,Z)$, let $\cP_B$ consist of the point $O_1$ at infinity in the direction of the $X$-axis together with the $22$ finite points
\begin{equation}\label{eq:points}
\begin{gathered}
(0,\pm\tau),\quad (0,\pm\tau^{-1}),\quad (\pm 2,\,0),\\
(\pm 1,\ \pm(\tau+\tau^{-1})),\quad (\pm 2\tau^{2},\ \pm\tau),\quad
(\pm 1,\ \pm(\tau-\tau^{-1})),\quad (\pm 2\tau^{-2},\ \pm\tau^{-1}),
\end{gathered}
\end{equation}
and let $\cL_B$ consist of the polar lines of the points of $\cP_B$ with respect to the conic $X^2-2Z^2=2$ (that is, the line $X_0X-2Z_0Z=2$ for a finite point $(X_0,Z_0)$, and the line $X=0$ for $O_1$).  Then $(\cP_B,\cL_B)$ is a geometric $(23_4)$ configuration.  It is defined over $\QQ(\sqrt{17})$, because the substitution $Z\mapsto\tau Z$ turns it into the configuration whose finite points are
\begin{equation}\label{eq:points17}
\begin{gathered}
(0,\pm1),\quad \Bigl(0,\pm\tfrac{\sqrt{17}-1}{4}\Bigr),\quad(\pm2,0),\quad \Bigl(\pm1,\pm\tfrac{3+\sqrt{17}}{4}\Bigr),\\
\Bigl(\pm\tfrac{1+\sqrt{17}}{2},\pm1\Bigr),\quad \Bigl(\pm1,\pm\tfrac{5-\sqrt{17}}{4}\Bigr),\quad \Bigl(\pm\tfrac{\sqrt{17}-1}{2},\pm\tfrac{\sqrt{17}-1}{4}\Bigr),
\end{gathered}
\end{equation}
self-polar with respect to the conic $X^2-\frac{1+\sqrt{17}}{2}Z^2=2$.  It is not projectively equivalent to a configuration with rational coordinates.
\end{enumerate}
The two configurations are not isomorphic as incidence structures.
\end{theorem}

Figures~\ref{fig:configA} and~\ref{fig:config} show the two configurations.  Combining Theorem~\ref{thm:main} with the results cited above gives the complete answer to the existence question for $k=4$.

\begin{corollary}\label{cor:complete}
A geometric $(n_4)$ configuration exists if and only if $n\ge 18$ and $n\ne 19$.
\end{corollary}

\begin{figure}[t]
\begin{minipage}[t]{0.485\textwidth}
\centering
\includegraphics[width=\textwidth,page=1]{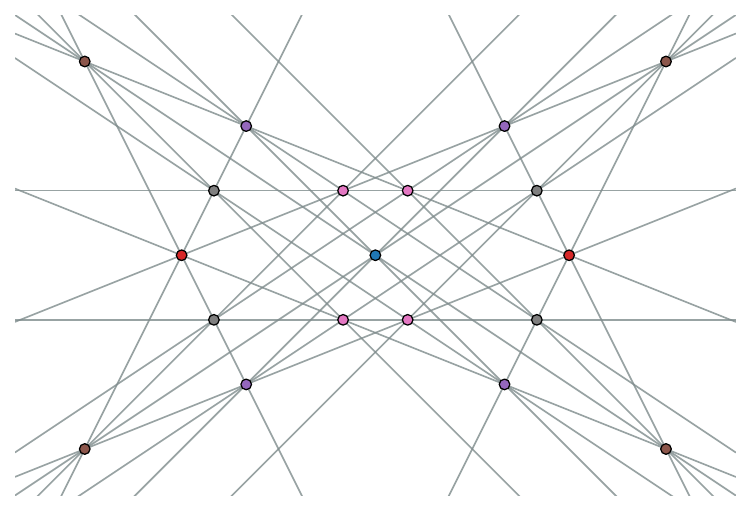}
\caption{The integral configuration $(\cP_A,\cL_A)$ of Theorem~\ref{thm:main}(A), drawn in the chart $x=1$ with coordinates $(y/x,\,z/x)$, so that $O_1=(1{:}0{:}0)$ is at the origin and the four points of $\cP_A$ on the line $x=0$ are at infinity.  Points are coloured by their orbit under the Klein four-group $V$; the eight orbits have sizes $1,2,2,2,4,4,4,4$.  The picture is symmetric under the reflections in both axes, which generate the full symmetry group. Figure created by the author using Matplotlib 3.10.8 \cite{matplotlib}.}
\label{fig:configA}
\end{minipage}\hfill
\begin{minipage}[t]{0.485\textwidth}
\centering
\includegraphics[width=\textwidth,page=2]{figs.pdf}
\caption{The self-polar configuration $(\cP_B,\cL_B)$ of Theorem~\ref{thm:main}(B), drawn in the chart of \eqref{eq:points}.  The point $O_1$ is at infinity in the horizontal direction, and the four horizontal lines are the four lines through it.  Points are coloured by their orbit under the Klein four-group $V$.  The picture is symmetric under the reflections $(X,Z)\mapsto(\pm X,\pm Z)$, which generate the full symmetry group of the configuration. Figure created by the author using Matplotlib 3.10.8 \cite{matplotlib}.}
\label{fig:config}
\end{minipage}
\end{figure}

Almost all known geometric $(n_4)$ configurations were found by imposing a large symmetry group (cyclic or dihedral of order at least $6$) and solving the resulting small system of equations; see \cite[Ch.~3]{Gru09} and \cite{BGP24}.  The following result shows that this strategy is bound to fail for $n=23$, and more generally for every prime $n\equiv3\pmod4$.

\begin{theorem}\label{thm:symmetry}
Let $n$ be a prime with $n\equiv3\pmod4$, and let $\cC$ be a geometric $(n_4)$ configuration.  Then the group of projective collineations of the real projective plane mapping $\cC$ onto itself is trivial, cyclic of order two, or the Klein four-group $\ZZ_2\times\ZZ_2$.
\end{theorem}

Both configurations of Theorem~\ref{thm:main} therefore have the largest possible symmetry group.  The third main result is that they are the only ones with this property.

\begin{theorem}\label{thm:classification}
Up to projective equivalence, the two configurations of Theorem~\ref{thm:main} are the only geometric $(23_4)$ configurations whose collineation group contains a Klein four-group.  The same statement holds for $(23_4)$ configurations in the complex projective plane.  Both configurations are self-polar, $(\cP_A,\cL_A)$ with respect to $x^2+6yz=0$ and $(\cP_B,\cL_B)$ with respect to the conic of Theorem~\ref{thm:main}(B).
\end{theorem}

Imposing exactly the symmetry allowed by Theorem~\ref{thm:symmetry}, a search over finite fields produced (B), and an exact enumeration of ``quotient structures'' produced (A) and confirms the answer; at $n=22$ and $n=26$ it recovers Cuntz's configurations and finds new ones.  But these are not the only $(23_4)$ configurations.  Applied to self-polar configurations with a single involution, the same method finds $50$ more, three of them with rational coordinates (Theorem~\ref{thm:selfpolar}).  Computing exactly three one-parameter families through them, we obtain infinitely many geometric $(23_4)$ configurations, among them infinitely many self-polar ones with rational coordinates and no symmetry (Theorem~\ref{thm:families}).

Theorem~\ref{thm:main} is proved in Section~\ref{sec:config} by exact computations in $\ZZ$ and in $\QQ(\tau)$, short enough to be checked by hand, Theorem~\ref{thm:symmetry} in Section~\ref{sec:symmetry} by counting orbits, and Theorem~\ref{thm:classification} in Section~\ref{sec:classification} by solving $4021$ polynomial systems exactly with Gr\"obner bases.

\section{Background}\label{sec:background}

We work in the real projective plane $\PP^2(\RR)$, whose points and lines are the one- and two-dimensional subspaces of $\RR^3$; a point is written $(x{:}y{:}z)$ and a line $[a{:}b{:}c]$, and the point lies on the line if and only if $ax+by+cz=0$.  Two points $(x{:}y{:}z)$ and $(x'{:}y'{:}z')$ coincide if and only if their cross product vanishes, and dually for lines.

\begin{definition}
A \emph{combinatorial $(n_k)$ configuration} is an incidence structure with $n$ points and $n$ lines in which every point is incident with $k$ lines, every line with $k$ points, and two points are incident with at most one common line.  Its \emph{Levi graph} is the bipartite graph on the points and lines whose edges are the incidences; it is $k$-regular of girth at least $6$.  A \emph{geometric $(n_k)$ configuration} is a combinatorial one whose points and lines are realized by distinct points and lines of $\PP^2(\RR)$ with exactly the prescribed incidences.
\end{definition}

The combinatorial existence problem is easy, since combinatorial $(n_4)$ configurations exist for all $n\ge13$.  The geometric problem is not, because a combinatorial configuration may have no realization by straight lines, and deciding realizability amounts to solving a system of polynomial equations.  In this paper ``configuration'' means geometric configuration unless stated otherwise.

\begin{definition}
Two configurations are \emph{projectively equivalent} if a projectivity of $\PP^2(\RR)$, i.e.\ an element of $\PGL_3(\RR)$, carries the points of one onto the points of the other and hence the lines onto the lines.  A \emph{collineation} of a configuration is a projectivity carrying its point set to itself; the collineations form a finite group, the \emph{symmetry group} of the configuration.  A \emph{correlation} is a projective map carrying points to lines and lines to points and preserving incidence; it is a \emph{polarity} if it is an involution, in which case it is the polarity with respect to a conic, given by a symmetric matrix.  A configuration is \emph{self-dual} if some correlation carries its points onto its lines and \emph{self-polar} if some polarity does.
\end{definition}

\begin{definition}
The \emph{Klein four-group} is the group $V=\ZZ_2\times\ZZ_2$.  Any Klein four-group of projectivities of $\PP^2(\RR)$ is conjugate to the group of diagonal matrices $\operatorname{diag}(\pm1,\pm1,\pm1)$, whose three nontrivial elements $\sigma_1,\sigma_2,\sigma_3$ are the harmonic homologies with centres the coordinate points $O_1,O_2,O_3$ and axes the opposite coordinate lines $\ell_1,\ell_2,\ell_3$.  A configuration is \emph{Klein-symmetric} if its symmetry group contains a Klein four-group.
\end{definition}

The \emph{cross-ratio} of four collinear points is invariant under projectivities, and a configuration whose points have coordinates in a field $K$ has all its cross-ratios in $K$; this is how fields of definition are bounded from below in Section~\ref{sec:config}.  The \emph{field of definition} of a configuration is the smallest field over which some projectively equivalent configuration has all its coordinates.

\section{The two configurations}\label{sec:config}

\subsection{The integral configuration}
The $23$ points \eqref{eq:pointsA} and the $23$ lines \eqref{eq:linesA} are given by primitive integer vectors.  A point lies on a line if and only if the corresponding integer dot product vanishes.  We label the points $O_1=(1{:}0{:}0)$, $T_1^{\pm}=(0{:}1{:}{\pm}1)$, $T_2^{\pm}=(0{:}3{:}{\pm}2)$, $S^{\pm}=(1{:}{\pm}1{:}0)$, $R_1^{\epsilon\delta}=(3{:}2\epsilon{:}2\delta)$, $R_2^{\epsilon\delta}=(2{:}3\epsilon{:}2\delta)$, $R_3^{\epsilon\delta}=(6{:}\epsilon{:}2\delta)$, $R_4^{\epsilon\delta}=(6{:}5\epsilon{:}2\delta)$ ($\epsilon,\delta\in\{+,-\}$).  Checking all incidences amounts to $23\cdot23=529$ integer dot products, and one finds that each line contains exactly four of the points, each point lies on exactly four of the lines, and no two points lie on two common lines.  For instance, the line $[2{:}6{:}-9]$ contains $T_2^{+}=(0{:}3{:}2)$ ($18-18=0$), $R_1^{++}=(3{:}2{:}2)$ ($6+12-18=0$), $R_3^{++}=(6{:}1{:}2)$ ($12+6-18=0$) and $R_4^{--}=(6{:}-5{:}-2)$ ($12-30+18=0$), and no other point of $\cP_A$.  The points are visibly pairwise distinct, and so are the lines.  This proves part (A) of Theorem~\ref{thm:main}; we also computed the same $529$ products with a short script.

The configuration is invariant under the Klein four-group $V=\{\operatorname{diag}(\pm1,\pm1,\pm1)\}$ (Section~\ref{sec:background}), whose orbits on $\cP_A$ have sizes $1,2,2,2,4,4,4,4$; by Theorem~\ref{thm:symmetry}, $V$ is the full collineation group of the configuration.  It is also self-polar.  The polarity $\pi_A\colon(x{:}y{:}z)\mapsto[x{:}3z{:}3y]$ of the conic $x^2+6yz=0$ maps $\cP_A$ onto $\cL_A$, as one reads off from \eqref{eq:pointsA} and \eqref{eq:linesA} (for instance $(3{:}2{:}2)\mapsto[1{:}2{:}2]$ and $(6{:}5{:}2)\mapsto[2{:}2{:}5]$).  So $\cL_A$ is simply the set of polars of the points \eqref{eq:pointsA}, and the configuration is determined by its points and one conic, exactly as $(\cP_B,\cL_B)$ is.

\subsection{Homogeneous description of the self-polar configuration}
The chart of Theorem~\ref{thm:main}(B) is $y=1$, $(X,Z)=(x/y,\,z/y)$.  The conic $X^2-2Z^2=2$ is the zero set of the quadratic form with Gram matrix $B=\operatorname{diag}(-1,2,2)$, and the polar line of a point $p$ is $\{q: p^{\mathsf T}Bq=0\}$.  Thus $\cL_B=\{\pi(p):p\in\cP_B\}$ where $\pi$ is the polarity of $B$, and a point $q\in\cP_B$ lies on the line $\pi(p)$ if and only if
\begin{equation}\label{eq:form}
-\,x_px_q+2y_py_q+2z_pz_q=0 .
\end{equation}
Since $B$ is symmetric, $q\in\pi(p)$ if and only if $p\in\pi(q)$.  Consequently the incidence structure $(\cP_B,\cL_B)$ is encoded by the graph $\Gamma$ on $\cP_B$ (with loops allowed) in which $p\sim q$ when \eqref{eq:form} holds, so that the four points on $\pi(p)$ are the neighbours of $p$, and $(\cP_B,\cL_B)$ is a $(23_4)$ configuration if and only if $\Gamma$ is $4$-regular (loops counted once) and two distinct vertices never have two common neighbours.  There are no loops here, i.e.\ no point of $\cP_B$ lies on its own polar.

\begin{table}[t]
\centering
\begin{tabular}{@{}llll@{}}
\toprule
label & point $(x:y:z)$ & label & point $(x:y:z)$\\
\midrule
$O_1$ & $(1:0:0)$ & $S^{\pm}$ & $(\pm2:1:0)$\\
$T_1^{\pm}$ & $(0:1:\pm\tau)$ & $C_1^{\epsilon\delta}$ & $(\epsilon:1:\delta(\tau+\tau^{-1}))$\\
$T_2^{\pm}$ & $(0:1:\pm\tau^{-1})$ & $C_2^{\epsilon\delta}$ & $(2\epsilon\tau^{2}:1:\delta\tau)$\\
 & & $C_3^{\epsilon\delta}$ & $(\epsilon:1:\delta(\tau-\tau^{-1}))$\\
 & & $C_4^{\epsilon\delta}$ & $(2\epsilon\tau^{-2}:1:\delta\tau^{-1})$\\
\bottomrule
\end{tabular}
\medskip
\caption{The $23$ points, $\epsilon,\delta\in\{+,-\}$.  Here $\tau^2=(1+\sqrt{17})/4$, $\tau^{-2}=(\sqrt{17}-1)/4$, $\tau^{-1}=\tau^3-\tau/2$, $(\tau+\tau^{-1})^2=(4+\sqrt{17})/2$ and $(\tau-\tau^{-1})^2=(\sqrt{17}-4)/2$. Table created by the author.}
\label{tab:points}
\end{table}

\subsection{Proof of Theorem \ref{thm:main}(B)}

\begin{proof}[Proof of Theorem~\ref{thm:main}(B)]
All coordinates in Table~\ref{tab:points} lie in the field $K=\QQ(\tau)$, which has degree $4$ over $\QQ$ (the minimal polynomial $2t^4-t^2-2$ is irreducible, its roots being $\pm\tau$ and $\pm i\tau^{-1}$; $K$ has two real embeddings and one pair of complex ones).  Evaluating the bilinear form \eqref{eq:form} on the $\binom{23}{2}+23=276$ pairs of points is a finite computation in $K$, in which every element is a polynomial of degree at most $3$ in $\tau$ with rational coefficients, reduced using $\tau^4=\tfrac12\tau^2+1$.  The result is an incidence table in which every point has exactly four neighbours, no point is its own neighbour, and any two distinct points have at most one common neighbour.  Moreover the $23$ points are pairwise distinct (they are given in normalized form, and the four generic orbits are distinguished already by the absolute values of their coordinates), and since $B$ is nondegenerate, distinct points have distinct polars.  It follows that $(\cP_B,\cL_B)$ is a geometric $(23_4)$ configuration, and its coordinates are real under the embedding $\tau\mapsto\sqrt{(1+\sqrt{17})/4}$.  The two configurations of Theorem~\ref{thm:main} are not isomorphic, since the automorphism groups of their Levi graphs (points and lines kept apart) have orders $4$ and $8$ respectively.  In Table~\ref{tab:points} every $x$- and $y$-coordinate lies in $\QQ(\sqrt{17})=\QQ(\tau^2)$ and every $z$-coordinate lies in $\tau\QQ(\sqrt{17})$, so the diagonal projectivity $z\mapsto z/\tau$ maps $\cP_B$ to the points \eqref{eq:points17}, maps $\cL_B$ to lines with coefficients in $\QQ(\sqrt{17})$, and maps the conic $X^2-2Z^2=2$ to $X^2-2\tau^2Z^2=2$, which is the conic stated; we verified all of this exactly.  Finally, the cross-ratio of the four points $(0,\pm1)$, $(0,\pm\frac{\sqrt{17}-1}{4})$ of \eqref{eq:points17} on the line $X=0$ is $33-8\sqrt{17}\notin\QQ$ (in fact $19$ of the $23$ collinear quadruples have irrational cross-ratio); since cross-ratios are projective invariants, no projectively equivalent configuration has all its coordinates in $\QQ$.  This completes the proof of Theorem~\ref{thm:main}(B).
\end{proof}

\subsection{Symmetry of the self-polar configuration}
In the chart, the Klein four-group $V$ of Section~\ref{sec:background} acts by the reflections $(X,Z)\mapsto(\pm X,\pm Z)$.  Table~\ref{tab:points} is visibly $V$-invariant, and $V$ commutes with $B$, so $V$ acts on $(\cP_B,\cL_B)$.  The $V$-orbits on $\cP_B$ have sizes $1,2,2,2,4,4,4,4$, namely the fixed point $O_1$, the pairs $T_1,T_2$ on the axis $\ell_1=\{x=0\}$ (each pair fixed by $\sigma_1$), the pair $S$ on the axis $\ell_3=\{z=0\}$, and the four regular orbits $C_1,\dots,C_4$.  By Theorem~\ref{thm:symmetry}, $V$ is the full collineation group of the configuration.  The automorphism group of the Levi graph with points and lines kept apart has order $8$ and contains $V$ with index $2$; the four automorphisms outside $V$ are induced by no collineation, as Theorem~\ref{thm:symmetry} demands (they come from the Galois conjugation $\sqrt{17}\mapsto-\sqrt{17}$).

\section{Symmetry groups of $(n_4)$ configurations}\label{sec:symmetry}

\begin{lemma}\label{lem:general}
Every geometric $(n_4)$ configuration $(\cP,\cL)$ contains four points in general position.
\end{lemma}

\begin{proof}
If all points of $\cP$ lay on one line $\ell$, every other line would contain at most one of them, so $\ell$ would be the only line of $\cL$.  Hence there are three non-collinear points $a,b,c\in\cP$.  If every point of $\cP$ lay on $ab$, $bc$ or $ca$, any other line would contain at most three points of $\cP$, so every line of $\cL$ would be one of these three, contradicting $n>3$.  Hence some fourth point of $\cP$ lies on none of them.
\end{proof}

We identify the real projective plane $\PP^2(\RR)$ with the sphere $S^2$ modulo antipodal identification.  A projective collineation of $\PP^2(\RR)$ is induced by an element of $\PGL_3(\RR)\cong\SL_3(\RR)$; a finite subgroup of $\SL_3(\RR)$ preserves a positive definite quadratic form (average any such form over the group) and is therefore conjugate to a finite subgroup of $\SO(3)$.  The finite subgroups of $\SO(3)$ are the cyclic groups $C_m$, the dihedral groups $D_m$ (of order $2m$), and the rotation groups $A_4$, $S_4$, $A_5$ of the Platonic solids.  Every group in this list other than $C_1$, $C_2$ and $D_2\cong\ZZ_2\times\ZZ_2$ contains a rotation of order $m\ge3$.  We will deduce Theorem~\ref{thm:symmetry} from Proposition~\ref{prop:norotation} below.

\begin{proposition}\label{prop:norotation}
Let $r\in\SO(3)$ be a rotation of order $m\ge3$, let $O$ and $\ell$ be the point and the line of $\PP^2(\RR)$ fixed by $r$, and let $(\cP,\cL)$ be an $r$-invariant geometric $(n_4)$ configuration.  Then $m\le n/2$, and one of the following holds.
\begin{enumerate}
\item[(a)] $O\in\cP$, $\ell\in\cL$, $m\in\{4,8\}$ and $n\equiv1\pmod4$, and moreover $n\equiv5\pmod8$ if $m=8$.
\item[(b)] $O\notin\cP$, $\ell\notin\cL$, and $m$ divides $n$ if $m$ is odd, while $m/2$ divides $n$ if $m$ is even.
\end{enumerate}
\end{proposition}

\begin{proof}
We first describe the orbits of $\langle r\rangle$.  The point $O$ is the image of the axis of $r$, and $\ell$ is the image of the great circle orthogonal to it.  Let $s=m$ if $m$ is odd and $s=m/2$ if $m$ is even.  Note that $r^j$ fixes a point of $\ell$ if and only if it maps that point to itself or to its antipode, that is, if and only if $2j\equiv0\pmod m$.  It follows that every orbit of $\langle r\rangle$ on the points of $\ell$ has size $s$, and by polarity the same holds for the lines through $O$.  Every other point and every other line has an orbit of size $m$, and such an orbit of points lies on the image of a circle of latitude with respect to the axis, which is a conic.

Let $a=[O\in\cP]$, let $b$ be the number of orbits of points of $\cP$ on $\ell$, and let $c$ be the number of remaining orbits of points of $\cP$; define $a'=[\ell\in\cL]$, $b'$ and $c'$ dually.  Counting orbit by orbit, we obtain
\begin{equation}\label{eq:count}
n=a+bs+cm=a'+b's+c'm.
\end{equation}

We first show that $c\ge2$.  A point of $\cP$ other than $O$ lies on four lines of $\cL$, of which at most one passes through $O$ and at most one is $\ell$, so there is a line $g\in\cL$ with $g\ne\ell$ and $O\notin g$.  The line $g$ meets $\ell$ in one point, and each of the $c$ remaining orbits lies on a conic and so meets $g$ in at most two points, so $4\le1+2c$, and hence $c\ge2$.  In particular $n\ge cm\ge2m$.

Suppose that $a=1$.  The four lines of $\cL$ through $O$ form a union of orbits of size $s$, so $s$ divides $4$.  Since $s=m\ge3$ is odd when $m$ is odd, $m$ is even and $s=m/2\in\{2,4\}$, so $m\in\{4,8\}$ and $b's=4$.  If $\ell\notin\cL$, then \eqref{eq:count} gives $n=4+c'm$, which is even, whereas $n=1+bs+cm$ is odd because $s$ and $m$ are even.  Hence $\ell\in\cL$, the four points of $\cP$ on $\ell$ form orbits of size $s$, so $bs=4$ and $n=5+cm$.  This is (a).  Conversely, if $\ell\in\cL$ then $bs=4$, so again $m$ is even, and $O\notin\cP$ would make $n=4+cm$ even but $n=1+b's+c'm$ odd; hence $\ell\in\cL$ implies $O\in\cP$.

Suppose finally that $a=0$.  Then $\ell\notin\cL$ by the previous paragraph, and $n=bs+cm$.  If $m$ is odd then $s=m$ divides $n$, and if $m$ is even then $n=\frac m2(b+2c)$.  This is (b), and the proof is complete.
\end{proof}

\begin{proof}[Proof of Theorem~\ref{thm:symmetry}]
Let $G$ be the group of projective collineations of $\PP^2(\RR)$ mapping $\cC$ onto itself.  By Lemma~\ref{lem:general}, $\cC$ contains four points in general position, and a projectivity fixing four points in general position is the identity.  Thus $G$ acts faithfully on the $n$ points of $\cC$, so $G$ is finite, and after applying a projectivity to $\cC$ we may assume $G\subset\SO(3)$.  Every finite subgroup of $\SO(3)$ other than $C_1$, $C_2$ and $D_2\cong\ZZ_2\times\ZZ_2$ contains a rotation of order $m\ge3$, so it suffices to show that $\cC$ is invariant under no such rotation.  Case (a) of Proposition~\ref{prop:norotation} requires $n\equiv1\pmod4$.  In case (b), either $m$ is odd and divides $n$ with $3\le m\le n/2$, or $m/2$ divides $n$ with $2\le m/2\le n/4$; either way $n$ is not prime.  This completes the proof of Theorem~\ref{thm:symmetry}.
\end{proof}

For every other $n\ge13$ some $m\ge3$ satisfies the conditions of Proposition~\ref{prop:norotation} ($m=4$ if $n$ is even or $n\equiv1\pmod4$, an odd prime divisor of $n$ otherwise), so the primes $n\equiv3\pmod4$ are exactly the values at which every rotation is excluded.  For $n\ge18$ they are $19,23,31,43,\dots$; there is no $(19_4)$ \cite{BP15}, the first $(31_4)$ and $(43_4)$ came from computer algebra and optimization \cite{BS13} and from gluing \cite{BP16}, and $23$ was the last case.  The cases $22$, $26$ and $37$ were also among the last, so Proposition~\ref{prop:norotation} is not the only difficulty.  The argument also applies to $5$-configurations.

\begin{proposition}\label{prop:five}
Let $n$ be a prime with $n\not\equiv1\pmod5$.  Then the collineation group of a geometric $(n_5)$ configuration is trivial, of order two, or a Klein four-group.
\end{proposition}

\begin{proof}
Lemma~\ref{lem:general} and the proof of Theorem~\ref{thm:symmetry} apply verbatim, and in the proof of Proposition~\ref{prop:norotation} the bound $5\le1+2c$ gives $c\ge2$ and $m\le n/2$.  If $O\in\cP$, the five lines through $O$ form orbits of size $s$, so $s=5$, and case (a) with congruences modulo $5$ in place of parity gives $\ell\in\cL$ and $n=6+cm\equiv1\pmod5$; otherwise $s$ divides $n$ with $2\le s\le n/2$.
\end{proof}

\section{The search method}\label{sec:search}

Theorem~\ref{thm:symmetry} suggested looking for configurations with the largest possible group, $V$, and, to shrink the search further, imposing a polarity (which reduces the number of unknowns from $18$ to $11$).  We first searched for $V$-symmetric self-polar $(23_4)$ configurations in the finite projective planes $\PP^2(\FF_p)$ for the primes $p=11,13,17,19,23$.  For each diagonal form and each choice of orbits, a constraint solver (Google OR-Tools CP-SAT \cite{ortools}) searched for a $V$-invariant set of $23$ points whose polar graph is $4$-regular without $4$-cycles.  This produced $17$ combinatorially distinct configurations.  Exactly one type occurred for more than one prime ($p=13,17,19$, always with a Levi graph automorphism group of order $8$), and it was the natural candidate for a configuration over a number field $K$, since such a configuration reduces modulo every prime of $K$ of degree one, so its type reappears for a positive proportion of primes, whereas a mod-$p$ accident does not.  Lifting confirmed this.  The polynomial system attached to the recurring type has eight complex solutions (two choices of sign of $\sqrt{-2}$, absorbed by a rescaling, times the four embeddings of $\QQ(\tau)$), all giving $(\cP_B,\cL_B)$, and the systems attached to the other $16$ types have no complex solutions.  (Consistently with the field $\QQ(\tau)$, the recurring type appeared exactly for those primes in the list for which $17$ is a square modulo $p$.)  The general classification of Section~\ref{sec:classification} then produced $(\cP_A,\cL_A)$, which the search had missed because it only used diagonal Gram matrices, whereas the Gram matrix of the polarity of $(\cP_A,\cL_A)$, $x^2+6yz$, interchanges $\sigma_2$ and $\sigma_3$.

\section{Klein-symmetric configurations}\label{sec:classification}

In Proposition~\ref{prop:structure} $\cC=(\cP,\cL)$ is a geometric $(n_4)$ configuration with $n\equiv3\pmod4$, over $\RR$ or over $\CC$, which is invariant under a Klein four-group $V$ of collineations, and after it $n=23$.  Choosing coordinates in which $V=\{\operatorname{diag}(\pm1,\pm1,\pm1)\}$, we write $\sigma_i$ for the element of $V$ with $+1$ in position $i$, $O_1,O_2,O_3$ for the three fixed points of $V$ (the coordinate points) and $\ell_i$ for the coordinate lines $\{x_i=0\}$.  A point is fixed by $\sigma_i$ if and only if it is $O_i$ or lies on $\ell_i$; every point not on $\ell_1\cup\ell_2\cup\ell_3$ has a regular orbit of size $4$, and every point of $\ell_i\setminus\{O_j,O_k\}$ has an orbit of size $2$ (a \emph{pair}, its ``mate'' being its image under $\sigma_j$).  Dually, a line is fixed by $\sigma_i$ if and only if it is $\ell_i$ or passes through $O_i$; lines through $O_i$ other than $\ell_j,\ell_k$ form pairs, and all other lines have regular orbits.  We write $c$ and $c'$ for the numbers of regular orbits of points and of lines in $\cC$, and $b_i$, $b'_i$ for the numbers of pairs on $\ell_i$ and through $O_i$.

\subsection{Structure}\label{sec:structure}
\begin{proposition}\label{prop:structure}
After renumbering the coordinates, the following hold.
\begin{enumerate}
\item[(a)] $O_1\in\cP$, $O_2,O_3\notin\cP$, $\ell_1\in\cL$, $\ell_2,\ell_3\notin\cL$; the line $\ell_1$ contains exactly two pairs $T_1,T_2$ of $\cP$, and the four lines of $\cL$ through $O_1$ are exactly the lines $O_1q$, $q\in T_1\cup T_2$.
\item[(b)] A point of $\cP$ on $\ell_i$, $i\in\{2,3\}$, other than $O_1$ lies on no line of $\cL$ through $O_1$, $O_2$ or $O_3$, and lies on exactly two pairs $\{g,\sigma_ig\}$ of regular lines of $\cL$, coming from two different regular orbits.  Dually, a line of $\cL$ through $O_2$ or $O_3$ other than $\ell_1$ contains no point of $\cP$ on $\ell_1$, $\ell_2$ or $\ell_3$, and contains exactly two pairs of regular points of $\cP$, from two different regular orbits.
\item[(c)] $b_2+b_3=\frac{n-5}2-2c$ and $b'_2+b'_3=\frac{n-5}2-2c'$, with $2b_2,2b_3\le c'$ and $2b'_2,2b'_3\le c$.  For $n=23$, $(c,c')\in\{(4,4),(4,3),(3,4)\}$; if $c=4$ then $b_2+b_3=1$, and if $c=3$ then $\{b_2,b_3\}=\{1,2\}$; dually for $c'$ and $b'_2,b'_3$.
\end{enumerate}
\end{proposition}

\begin{proof}
Let $a$ and $a'$ be the numbers of coordinate points in $\cP$ and of coordinate lines in $\cL$.  We first determine $a$ and $a'$.  Counting the points of $\cP$ and the lines of $\cL$ orbit by orbit gives
\[
n=a+2\sum_i b_i+4c=a'+2\sum_i b'_i+4c',
\]
so $a$ and $a'$ are odd.  Suppose that $O_i\in\cP$.  The lines of $\cL$ through $O_i$ are $\ell_j$ and $\ell_k$, if these lie in $\cL$, together with pairs of lines interchanged by $\sigma_j$ or $\sigma_k$; hence
\[
4\equiv[\ell_j\in\cL]+[\ell_k\in\cL]\pmod 2 .
\]
Suppose for sake of contradiction that $a=3$.  Then the congruence above holds for all three indices $i$, forcing the three indicators $[\ell_i\in\cL]$ to be equal.  Since $a'$ is odd, all three are equal to $1$, so $a'=3$ and $b'_i=1$ for every $i$; but then $n=3+6+4c'\equiv1\pmod4$, which is impossible.  It follows that $a=1$, say $O_1\in\cP$, and by the dual argument $a'=1$.  Since exactly one $\ell_i$ lies in $\cL$, the congruence at $O_1$ forces $\ell_2,\ell_3\notin\cL$ and $\ell_1\in\cL$, and then $4=2b'_1$ gives $b'_1=2$; the dual argument applied to $\ell_1$ gives $b_1=2$.

We now prove (a).  Let $q$ be one of the four points of $\cP$ on $\ell_1$.  Since $O_2,O_3\in\ell_1$, the only lines through $q$ that are not regular are $\ell_1$ and the line $O_1q$.  Note that a regular line $g$ through $q$ comes with the line $\sigma_1g\ne g$, which also passes through $q$, and that the orbit of $g$ contains no further line through $q$, since such a line would pass through $q$ and $\sigma_2q$ and hence coincide with $\ell_1$.  Counting the lines of $\cL$ through $q$, we obtain $4=1+[O_1q\in\cL]+2k$ for some integer $k$, whence $O_1q\in\cL$.  Thus the four lines $O_1q$, for the four points $q$ of $\cP$ on $\ell_1$, are lines of $\cL$ through $O_1$; as $\cL$ contains exactly four lines through $O_1$, these are all of them.  This proves (a).

We next prove (b).  Let $s\in\cP\cap\ell_3$, $s\ne O_1$.  The lines through $s$ are $\ell_3\notin\cL$, the line $O_3s$, and regular lines.  The regular lines through $s$ come in pairs $\{g,\sigma_3g\}$, with at most one pair from each regular orbit, since two pairs from one orbit would give a line through $s$ and $\sigma_1s$, that is, the line $\ell_3$.  Lines through $O_1$ or $O_2$ other than $\ell_3$ meet $\ell_3$ only in $O_1$ or $O_2$, so none of them passes through $s$.  Counting, we obtain $4=[O_3s\in\cL]+2k$, whence $O_3s\notin\cL$ and $k=2$.  Thus the point $s$ lies on two pairs of regular lines coming from two different regular orbits.  The same argument applies to the points of $\cP$ on $\ell_2$ and, dually, to the lines of $\cL$ through $O_2$ and $O_3$.  This proves (b).

It remains to prove (c).  A regular line meets $\ell_3$ in exactly one point, so a regular orbit of lines is attached, in the sense of (b), to at most one pair of $\cP$ on $\ell_3$.  Therefore $2b_3\le c'$, and likewise $2b_2\le c'$, $2b'_3\le c$ and $2b'_2\le c$.  Substituting $a=a'=1$ and $b_1=b'_1=2$ into the count above gives the two equations of (c).  For $n=23$ they read $b_2+b_3=9-2c$ and $b'_2+b'_3=9-2c'$, so $c,c'\le4$.  Combining these with the inequalities we obtain $9-2c\le c'$ and $9-2c'\le c$, whence $c,c'\ge3$.  If $c=3$, then $b_2+b_3=3$ with $b_2,b_3\le c'/2\le2$, so $\{b_2,b_3\}=\{1,2\}$ and $c'=4$.  If $c=4$, then $b_2+b_3=1$.  The dual statements hold with the roles of $c$ and $c'$ interchanged.  Thus $(c,c')\in\{(4,4),(4,3),(3,4)\}$, and the proof is complete.
\end{proof}

Thus for every $n\equiv3\pmod4$, in particular for every prime $n\equiv3\pmod4$, where by Theorem~\ref{thm:symmetry} the Klein four-group is the largest possible symmetry, a Klein-symmetric $(n_4)$ configuration has the same rigid shape, with one fixed point and one fixed line in the configuration, two pairs on that line and an odd number of pairs on the other two axes.  The method below applies to these $n$ with a longer list of cases, and the next primes of this form are $31$ and $43$.

Now let $n=23$.  By duality (the standard correlation $(x{:}y{:}z)\mapsto[x{:}y{:}z]$ commutes with $V$ and maps Klein-symmetric $(23_4)$ configurations to Klein-symmetric $(23_4)$ configurations) we may assume $c=4$; by the coordinate interchange $y\leftrightarrow z$ we may assume, in the case $c'=3$, that $(b'_2,b'_3)=(1,2)$.  Thus there are four cases to consider.  In two of them $c'=4$, with the pair of $\cP$ on $\ell_3$ and the pair of $\cL$ through $O_3$ or through $O_2$ (the remaining two combinations are obtained by $y\leftrightarrow z$), and in the other two $c'=3$, with the pair of $\cP$ on $\ell_3$ or on $\ell_2$.  We write $T_i=\{q_i,q_i'\}$ ($q_i'=\sigma_2q_i$), $m_i=O_1q_i$, $m_i'=O_1q_i'$, $S=\{s,s'\}$ for the pair of $\cP$ on $\ell_2\cup\ell_3$, $P_j=\{\alpha p_j:\alpha\in V\}$ for the regular point orbits and $G_a=\{\alpha g_a:\alpha\in V\}$ for the regular line orbits, and $n_k,n_k'$ for the pairs of lines through $O_2$ or $O_3$.

\begin{lemma}[local incidence rules]\label{lem:local}
\begin{enumerate}
\item[(i)] $O_1$ lies exactly on $m_1,m_1',m_2,m_2'$; each $q_i$ lies on $\ell_1$, on $m_i$, and on exactly one pair $\{g,\sigma_1g\}$ of regular lines, from one orbit; the orbits attached to $T_1$ and to $T_2$ are different.  Dually, each $m_i$ contains $O_1$, $q_i$ and exactly one pair $\{p,\sigma_1p\}$ of regular points, and the orbits attached to $\{m_1,m_1'\}$ and to $\{m_2,m_2'\}$ are different.
\item[(ii)] Each point of $S$ lies on no line of any pair $\{n_k,n_k'\}$ and on exactly two pairs $\{g,\sigma_Sg\}$ of regular lines from two different orbits, where $\sigma_S$ is the element of $V$ fixing $S$ pointwise.  Dually for the lines $n_k$.
\item[(iii)] A regular point lies on at most one line of each regular orbit and on at most one line through each $O_i$; a regular line contains at most one point of each regular orbit and at most one point on each $\ell_i$.
\item[(iv)] All degrees equal $4$, and no two points lie on two common lines.
\end{enumerate}
\end{lemma}

\begin{proof}
Statement (i) was established in the proof of Proposition~\ref{prop:structure}(a), together with the fact that two points of $\ell_1$ never lie on a common regular line, which is what separates the orbits attached to $T_1$ from those attached to $T_2$; statement (ii) is Proposition~\ref{prop:structure}(b).

We now prove (iii).  Let $p$ be a regular point.  Then $p$ lies on the unique line $O_ip$ through $O_i$, so $p$ lies on at most one line of $\cL$ through each centre.  Suppose for sake of contradiction that $p$ lies on two lines $g$ and $\alpha g$ of one regular orbit, where $\alpha\in V\setminus\{e\}$.  Then $g$ contains both $p$ and $\alpha p$, and hence passes through the fixed point of $\alpha$, contradicting the regularity of $g$.  Thus $p$ lies on at most one line of each regular orbit, and the statement for regular lines is dual.  This proves (iii).

Lastly, (iv) is part of the definition of an $(n_4)$ configuration.
\end{proof}

\subsection{Enumeration of quotient structures}\label{sec:enum}
A \emph{quotient structure} is an incidence structure on the $46$ symbols $O_1$, $q_i,q_i'$, $s,s'$, $\alpha p_j$, $\ell_1$, $m_i,m_i'$, $n_k,n_k'$, $\alpha g_a$ that satisfies Lemma~\ref{lem:local} and is invariant under the action of $V$ on the symbols.  Such a structure is determined by which regular orbits are attached to $T_1,T_2$, to $\{m_1,m_1'\},\{m_2,m_2'\}$, to $S$ and to the pairs $\{n_k,n_k'\}$, together with the ``coset'' data saying which two of the four lines (points) of an attached orbit pass through (lie on) a given point (line); and, for each incident pair $(P_j,G_a)$ of regular orbits, a \emph{voltage} $\alpha\in V$ meaning $\beta p_j\in\beta\alpha g_a$ for all $\beta\in V$.  The degrees of Lemma~\ref{lem:local}(iv) determine how many regular orbits of lines each $P_j$ is incident with (namely $4$ minus the number of attachments of $P_j$), and dually.  The choice of representatives $q_i,s,p_j,g_a,n_k$ is used to normalize the coset data and the voltages on a spanning forest of the graph of incident pairs $(P_j,G_a)$; the relabellings $T_1\leftrightarrow T_2$ (together with $P_1\leftrightarrow P_2$ and $G_1\leftrightarrow G_2$), $P_3\leftrightarrow P_4$, $G_3\leftrightarrow G_4$ are used to normalize the attachments; the remaining redundancy is removed by computing with \texttt{nauty} \cite{MP14} a canonical form of the coloured graph obtained from the Levi graph of the structure by adjoining, for $i=1,2,3$, a vertex of colour $i$ for every pair $\{v,\sigma_iv\}$ of points or of lines with $\sigma_iv\ne v$; two quotient structures give the same system of equations up to renaming exactly when these coloured graphs are isomorphic.  The condition (iv) that no two points lie on two common lines is imposed incrementally while the voltages are assigned, which prunes the search drastically.  The enumeration produces $61{,}240$ raw structures, which \texttt{nauty} sorts into $4021$ classes; since a wrong identification would silently lose a case, for each of the $57{,}219$ identifications we checked with independent code that the canonical relabellings make the two coloured graphs identical, so every identification is backed by an explicit isomorphism.

\begin{proposition}\label{prop:enum}
The numbers of inequivalent quotient structures in the four cases are $1361$ ($c'=4$, $S\subset\ell_3$, $\{n_1,n_1'\}$ through $O_3$), $1286$ ($c'=4$, $S\subset\ell_3$, $\{n_1,n_1'\}$ through $O_2$), $699$ and $675$ ($c'=3$, $S\subset\ell_3$ resp.\ $S\subset\ell_2$), $4021$ in all.
\end{proposition}

\subsection{The polynomial systems}
Fix a quotient structure and write $q_1=(0{:}1{:}t_1)$, $q_2=(0{:}1{:}t_2)$, $s=(1{:}S{:}0)$ or $(1{:}0{:}S)$, $p_j=(1{:}u_j{:}v_j)$, $g_a=[1{:}m_a{:}n_a]$, $n_k=[1{:}S'_k{:}0]$ or $[1{:}0{:}S'_k]$; the lines $m_i=[0{:}t_i{:}-1]$ are determined.  The diagonal rescalings of the coordinates commute with $V$ and make it possible to normalize $t_1=S=1$, leaving $18$ unknowns.  Each incidence of the quotient structure is a bilinear equation; there are $20$ distinct ones.  A solution gives a $(23_4)$ configuration with incidence structure the given quotient structure if and only if the $23$ points and the $23$ lines are distinct and no incidence outside the structure occurs.  We computed with \textsc{Singular} \cite{Singular} a Gr\"obner basis of the ideal $I$ of the equations, saturated with respect to the distinctness conditions (all unknowns nonzero, $t_2^2\ne1$, $(u_j^2,v_j^2)\ne(u_k^2,v_k^2)$, $(m_a^2,n_a^2)\ne(m_b^2,n_b^2)$, $S_k'^2\ne S_l'^2$ for pairs through the same $O_i$), and solved the resulting ideal $J$ when it is zero-dimensional.

\begin{proposition}\label{prop:solve}
Of the $4021$ quotient structures, $4018$ have no complex solutions, since the saturated ideal $J$ is the unit ideal.  The remaining three are as follows.
\begin{enumerate}
\item[(a)] a structure with $S\subset\ell_3$, $\{n_1,n_1'\}$ through $O_3$, whose two solutions both satisfy $S'_1=1$, which produces an incidence not in the structure; saturating by $1-S'_1$ gives the unit ideal, so it has no realization;
\item[(b)] a structure with $S\subset\ell_3$, $\{n_1,n_1'\}$ through $O_3$, with exactly two nondegenerate solutions, both real and with coordinates in $\QQ(\sqrt{17})$; they are conjugate under $\sqrt{17}\mapsto-\sqrt{17}$, and in the normalization $u_1=v_1=1$ they are carried onto the $\QQ(\sqrt{17})$-model \eqref{eq:points17} of $(\cP_B,\cL_B)$ by the diagonal projectivities $\operatorname{diag}\bigl(1,\tfrac{1-\sqrt{17}}{8},\tfrac{1-\sqrt{17}}{8}\bigr)$ and $\operatorname{diag}\bigl(1,-\tfrac{1+\sqrt{17}}{8},-\tfrac12\bigr)$, which we verified on all $23$ points by exact arithmetic;
\item[(c)] a structure with $S\subset\ell_3$, $\{n_1,n_1'\}$ through $O_2$, with exactly one solution, which is rational and nondegenerate, namely
$t_2=u_1=v_1=\tfrac23$, $u_2=-\tfrac32$, $v_2=-1$, $u_3=\tfrac16$, $u_4=-\tfrac56$, $v_3=v_4=-\tfrac13$, $S'_1=3$, $(m_a)=(-2,3,-1,-1)$, $(n_a)=(2,-\tfrac92,-\tfrac12,\tfrac52)$; clearing denominators gives exactly the points \eqref{eq:pointsA} and lines \eqref{eq:linesA}.
\end{enumerate}
\end{proposition}

Every statement in Proposition~\ref{prop:solve} is thus established by exact computation.  For each of the $4018$ empty systems we also extracted from Singular an explicit certificate $1=\sum_i a_if_i$ with polynomial coefficients $a_i$, and for the structure in (a) a certificate that the polynomial expressing its forbidden incidence lies in its ideal; all $4019$ certificates were re-verified by exact polynomial arithmetic in a standard-library Python \cite{python} script.  Since each certificate is checked against the system of its own quotient structure, this part of the proof depends neither on a Gr\"obner basis implementation nor on the enumeration code.  All $4021$ Gr\"obner basis computations were repeated independently in Macaulay2 \cite{M2}, with the same saturations; the two systems agree on the dimension and, where zero-dimensional, on the degree of every saturated ideal.

\subsection{Proof of Theorem \ref{thm:classification}}

\begin{proof}[Proof of Theorem~\ref{thm:classification}]
Existence is Theorem~\ref{thm:main}.  We prove uniqueness.  Let $\cC$ be a geometric $(23_4)$ configuration over $\RR$ or $\CC$ with a Klein four-group $V$ of collineations.  We first reduce to one of the quotient structures.  By Proposition~\ref{prop:structure} and Lemma~\ref{lem:local}, after a change of coordinates commuting with $V$, an interchange of $y$ and $z$, and possibly passing to the dual configuration, the incidence structure of $\cC$ is one of the $4021$ quotient structures of Proposition~\ref{prop:enum}, and the points and lines of $\cC$ give a solution of the corresponding polynomial system satisfying all of the side conditions.

We next identify the solution.  By Proposition~\ref{prop:solve}, this solution is one of the three nondegenerate solutions in (b) and (c), which are projectively equivalent to $(\cP_B,\cL_B)$ or equal to $(\cP_A,\cL_A)$.  Since both configurations are self-polar, passing to the dual configuration produces nothing new.  It remains to observe that the equivalences are realized by real projectivities when $\cC$ is real.  By Lemma~\ref{lem:general}, $\cC$ contains four real points in general position, and a projectivity mapping four real points in general position to four real points is real.  Thus $\cC$ is projectively equivalent over $\RR$ to $(\cP_A,\cL_A)$ or $(\cP_B,\cL_B)$.

The statements about polarities were established in Section~\ref{sec:config}.  This completes the proof of Theorem~\ref{thm:classification}.
\end{proof}

Theorem~\ref{thm:main} rests on exact computations that can be redone by hand, and Theorem~\ref{thm:symmetry}, Proposition~\ref{prop:structure}, Lemma~\ref{lem:local} and Proposition~\ref{prop:involution} are proved by hand.  Theorem~\ref{thm:classification} depends on the enumeration programs, on \texttt{nauty} and on exact Gr\"obner basis computations, repeated in Macaulay2, whose empty cases are certified as described above.  Independently of the enumeration, both incidence structures are rigid.  Sending four points in general position (Lemma~\ref{lem:general}) to the standard frame, the incidences determine the other points and lines as joins and intersections, except for three or four lines through known points, each parametrized in two affine charts, and the unused incidences give polynomial equations in these parameters, together with $zF=1$ for a product $F$ of squared norms of coordinate vectors.  Gr\"obner bases over $\QQ$ show that only one chart is nonempty, with exactly one real solution for $(\cP_A,\cL_A)$ (one even over $\CC$) and two for $(\cP_B,\cL_B)$, conjugate under $\sqrt{17}\mapsto-\sqrt{17}$ and differing by a relabelling from a Levi graph automorphism not induced by a collineation; a second choice of frame gives the same result.  So each configuration is determined up to projective equivalence by its incidence structure, and the proofs of Theorems~\ref{thm:main}, \ref{thm:symmetry} and~\ref{thm:classification} contain no floating-point computation.

\section{The classification at $n=22$ and $n=26$}\label{sec:22}

The machinery of Section~\ref{sec:classification} applies verbatim to other values of $n$.  We ran it for $n=22$ and $n=26$ because at each the literature supplies a Klein-symmetric configuration that the enumeration must return, so a run which omitted structures would be detected.  Let $\cC$ be a geometric $(22_4)$ configuration, over $\RR$ or $\CC$, invariant under a Klein four-group $V$, with the notation of Section~\ref{sec:classification}.

\begin{proposition}\label{prop:22structure}
No coordinate point $O_i$ belongs to $\cP$ and no coordinate line $\ell_i$ belongs to $\cL$.  Every pair-point on $\ell_i$ lies on no pair-line and on two $\sigma_i$-pairs of regular lines from two different regular orbits; every pair-line through $O_i$ contains no pair-point and two $\sigma_i$-pairs of regular points from two different orbits; a regular orbit of lines is attached in this way to at most one pair on each axis, and dually.  Consequently
\[
22=2(b_1+b_2+b_3)+4c=2(b_1'+b_2'+b_3')+4c'
\]
with $2b_i\le c'$ and $2b_i'\le c$, so that $(c,c')\in\{3,4,5\}^2\setminus\{(3,3)\}$.
\end{proposition}

\begin{proof}
Counting orbits as in the proof of Proposition~\ref{prop:structure}, the number $a$ of coordinate points in $\cP$ is even, so $a\in\{0,2\}$, and likewise $a'\in\{0,2\}$.  Suppose $a=2$, say $O_1,O_2\in\cP$.  The parity conditions $4\equiv[\ell_j\in\cL]+[\ell_k\in\cL]\pmod2$ at $O_1$ and at $O_2$ force the three indicators $[\ell_i\in\cL]$ to be equal, hence all zero since $a'$ is even; then $b_1'=b_2'=2$.  A pair-line $m$ through $O_1$ contains $O_1$, so $4=1+[m\cap\ell_1\in\cP]+2k$ and the point $q=m\cap\ell_1$ lies in $\cP$; it is a pair-point on $\ell_1$ (it is not $O_2$ or $O_3$, since $m$ is not an axis).  But the lines of $\cL$ through a pair-point $q\in\ell_1$ are the line $O_1q$ (if in $\cL$) and regular lines in $\sigma_1$-pairs, so $4=[O_1q\in\cL]+2k'$ gives $O_1q\notin\cL$, contradicting $O_1q=m\in\cL$.  Hence $a=0$ and dually $a'=0$.

The remaining assertions are proved exactly as in Proposition~\ref{prop:structure} and Lemma~\ref{lem:local}, with the counting $22=2\sum b_i+4c$.  From $\sum b_i=11-2c\le3\lfloor c'/2\rfloor$ and $\sum b_i'=11-2c'\le3\lfloor c/2\rfloor$ one gets $3\le c,c'\le5$, and $(c,c')=(3,3)$ is excluded because it would need $b_i\le1$ with $\sum b_i=5$.
\end{proof}

Up to permuting the coordinates there are $25$ patterns $(c,c';b_1,b_2,b_3;b_1',b_2',b_3')$.  The enumeration of quotient structures, with the same gauge fixing, four-cycle pruning and \texttt{nauty} canonical forms as in Section~\ref{sec:classification}, produced $28{,}711$ of them, each giving $22$ bilinear equations in $20$ unknowns after normalizing $u_1=v_1=1$.  Exactly $39$ structures have a nonempty saturated solution set; they carry $80$ nondegenerate complex solutions, of which $38$ are real up to a diagonal rescaling.  Sorting the real solutions into projective equivalence classes, and identifying the coordinates as algebraic numbers and verifying the resulting configurations by exact arithmetic in the corresponding number fields, gives the following.

\begin{theorem}\label{thm:22}
Up to projective equivalence there are exactly nine geometric $(22_4)$ configurations with a Klein four-group of symmetries, of five combinatorial types, as follows.
\begin{enumerate}
\item[(i)] the two configurations of Cuntz \cite{Cun18} (the parameter $w=(-7\pm3\sqrt{17})/2$), which are combinatorially isomorphic but projectively inequivalent, and their two duals;
\item[(ii)] a type $T_3$ with one configuration, defined over $\QQ(\sqrt3)$, whose points and lines are the $V$-orbits of
\begin{equation}\label{eq:T3}
\begin{gathered}
(0,1,2{+}\sqrt3),\quad (1,0,1{-}\tfrac13\sqrt3),\quad (1,-1{+}\tfrac13\sqrt3,0),\quad (1,1,1),\\
(1,2{-}\sqrt3,2{-}\sqrt3),\quad (1,-1{+}\tfrac23\sqrt3,1),\quad (1,2{-}\sqrt3,-\tfrac13\sqrt3);\\[2pt]
[0,1,-1],\quad [1,0,-1],\quad [1,-2{-}\sqrt3,0],\quad [1,\tfrac92{+}\tfrac52\sqrt3,-\tfrac32{-}\tfrac12\sqrt3],\\
[1,\tfrac32{+}\tfrac12\sqrt3,-\tfrac32{+}\tfrac12\sqrt3],\quad [1,\tfrac12{+}\tfrac12\sqrt3,-\tfrac32{-}\tfrac12\sqrt3],\quad [1,\tfrac32{+}\tfrac12\sqrt3,-\tfrac12{-}\tfrac12\sqrt3]
\end{gathered}
\end{equation}
(the first three points and lines have orbits of size two, the others of size four);
\item[(iii)] a type $T_1$ with two configurations, conjugate under $\sqrt\varphi\mapsto-\sqrt\varphi$ and projectively inequivalent, defined over the quartic field $\QQ(\sqrt5,\sqrt\varphi)$, $\varphi=(1+\sqrt5)/2$;  writing $t=\sqrt\varphi$, so that $t^4=t^2+1$, the coordinates of both are polynomials in $t$ with integer coefficients;
\item[(iv)] a type $T_2$ with two configurations, projectively inequivalent, conjugate under $\sqrt2\mapsto-\sqrt2$, defined over $\QQ(\sqrt2,\sqrt5)$.
\end{enumerate}
The same nine are the real ones among the $80$ Klein-symmetric $(22_4)$ configurations in the complex projective plane.  None of the nine is definable over $\QQ$, the configurations of type $T_1,T_2$ are not definable over $\QQ(\sqrt5)$, and all nine have collineation group exactly $V$; the Levi graphs of $T_2$ and $T_3$ have automorphism groups of order $8$, those of $T_1$ and of Cuntz's types of order $4$.
\end{theorem}

The statements about fields follow from cross-ratios of collinear quadruples, computed exactly.  For $T_3$, $20$ of the $22$ cross-ratios are irrational; for $T_1$ and $T_2$ only $10$ resp.\ $8$ of them lie in $\QQ(\sqrt5)$ and none in $\QQ$.  The enumeration recovered Cuntz's configurations, the only $(22_4)$ configurations in the literature, and the Levi graphs of $T_1,T_2,T_3$ are not isomorphic to theirs (checked with \texttt{nauty}), so these three types are new.

\subsection{The case $n=26$}\label{sec:26}

For $n=26$ the enumeration produces $2{,}348{,}674$ quotient structures in twenty-two patterns, all solved exactly with no system left unresolved.  They carry $244$ real nondegenerate solutions, each checked to give a geometric $(26_4)$ configuration, and a further $346$ defined only over $\CC$.

\begin{theorem}\label{thm:26}
Up to projective equivalence there are exactly $160$ geometric $(26_4)$ configurations with a Klein four-group of symmetries, falling into $95$ combinatorial types, every one with collineation group exactly $V$.
\end{theorem}

Cuntz's $(26_4)$ configuration \cite{Cun18}, the only one in the literature, is Klein-symmetric and is returned by the enumeration as one of the $95$ types, so the remaining $94$ are new.  None of the $95$ is definable over $\QQ$, and a majority require a field of degree five, in contrast with $n=22$ and $n=23$, where every Klein-symmetric configuration is defined over a field of degree at most four.

Exactly one of the types is self-dual, and it is more than self-dual.  All four of its combinatorial dualities are involutions, and the configuration admits four projective polarities, obtained from one another by composition with $V$, of which exactly one is definite.  It is therefore self-polar, like $(\cP_B,\cL_B)$, but with respect to a conic without real points.  Cuntz's $(26_4)$ is not even self-dual, since its Levi graph has the same automorphism group of order $4$ whether or not points and lines may be interchanged, so this is a different configuration.

\section{Configurations with an involution}\label{sec:remarks}

Theorem~\ref{thm:classification} classifies the Klein-symmetric $(23_4)$ configurations, and by Theorem~\ref{thm:symmetry} every other one has a single symmetry of order two or none.  The counting arguments of Section~\ref{sec:classification} apply to the first of these and reduce it to nine orbit patterns.

\begin{proposition}\label{prop:involution}
Let $\cC=(\cP,\cL)$ be a geometric $(23_4)$ configuration invariant under an involutory collineation $\sigma$, a harmonic homology with centre $O$ and axis $\ell$.  Let $a=[O\in\cP]$, $a'=[\ell\in\cL]$, let $b$ be the number of points of $\cP$ on $\ell$ and $b'$ the number of lines of $\cL$ through $O$, and let $c$ and $c'$ be the numbers of pairs $\{p,\sigma p\}$ of points of $\cP$ and $\{g,\sigma g\}$ of lines of $\cL$ not fixed by $\sigma$.  Then exactly one of the following holds.
\begin{enumerate}
\item[(a)] $a=a'=1$, $b=b'=4$ and $c=c'=9$.  Every point of $\cP$ on $\ell$ lies on a line of $\cL$ through $O$, every line of $\cL$ through $O$ meets $\ell$ in a point of $\cP$, and these correspondences are inverse bijections.
\item[(b)] $a=a'=0$, $b$ and $b'$ are odd, $(b,b')$ is one of the eight pairs
\[
(1,1),\ (1,3),\ (1,5),\ (3,1),\ (3,3),\ (3,5),\ (5,1),\ (5,3),
\]
and $c=(23-b)/2$, $c'=(23-b')/2$.  No line of $\cL$ through $O$ meets $\ell$ in a point of $\cP$; every point of $\cP$ on $\ell$ lies on exactly two pairs of lines of $\cL$, and every line of $\cL$ through $O$ contains exactly two pairs of points of $\cP$.
\end{enumerate}
\end{proposition}

\begin{proof}
The fixed points of $\sigma$ are $O$ and the points of $\ell$; the fixed lines are $\ell$ and the lines through $O$.  Counting points and lines orbit by orbit gives $23=a+b+2c=a'+b'+2c'$, so $a+b$ and $a'+b'$ are odd.

Let $g$ be a line of $\cL$ through $O$.  It is fixed by $\sigma$, and $\sigma$ acts on it as an involution with fixed points $O$ and $g\cap\ell$, so the points of $\cP$ on $g$ other than these come in pairs $\{p,\sigma p\}$.  Hence
\begin{equation}\label{eq:linethroughO}
4=a+[\,g\cap\ell\in\cP\,]+2k
\end{equation}
for some integer $k\ge0$.  Dually, for a point $q$ of $\cP$ on $\ell$, the lines of $\cL$ through $q$ other than $\ell$ and $Oq$ come in pairs, so
\begin{equation}\label{eq:pointonell}
4=a'+[\,Oq\in\cL\,]+2k'
\end{equation}
for some integer $k'\ge0$.

Suppose first that $a=1$.  The four lines of $\cL$ through $O$ are fixed, so $b'=4$, $a'=1$ as $a'+b'$ is odd, and $b=4$ as $\ell\in\cL$.  Then $23=1+4+2c$ gives $c=9$, and dually $c'=9$.  By \eqref{eq:linethroughO} with $a=1$, each line $g$ of $\cL$ through $O$ satisfies $g\cap\ell\in\cP$, and by \eqref{eq:pointonell} with $a'=1$, each point $q$ of $\cP$ on $\ell$ satisfies $Oq\in\cL$; the maps $g\mapsto g\cap\ell$ and $q\mapsto Oq$ are injective and inverse to each other.  This is case (a).

Suppose next that $a=0$ and $a'=1$.  Then $\ell\in\cL$ contains exactly four points of $\cP$, so $b=4$ and $a+b$ is even, a contradiction.  So $a=0$ forces $a'=0$, and $b$, $b'$ are odd.  By \eqref{eq:linethroughO} with $a=0$, every line $g$ of $\cL$ through $O$ has $g\cap\ell\notin\cP$ and contains exactly two pairs of points of $\cP$; by \eqref{eq:pointonell} with $a'=0$, every point $q$ of $\cP$ on $\ell$ has $Oq\notin\cL$ and lies on exactly two pairs $\{g,\sigma g\}$ of lines of $\cL$.  A line $g$ not fixed by $\sigma$ does not pass through $O$, and $g\cap\ell$ is fixed by $\sigma$, so $g\cap\ell=\sigma g\cap\ell$ is the unique common point of $g$ and $\sigma g$; thus each pair of lines is attached to exactly one point of $\ell$, and the $2b$ pairs attached to the $b$ points of $\cP$ on $\ell$ are distinct.  Hence $2b\le c'=(23-b')/2$, that is, $4b+b'\le23$, and dually $4b'+b\le23$.  The odd positive solutions are exactly the eight pairs in (b).  This is case (b).
\end{proof}

For self-polar configurations the patterns of Proposition~\ref{prop:involution} can be enumerated completely.  Let $\cC$ be self-polar with respect to a polarity $\pi$ and have a nontrivial collineation group $G$.  Conjugation by $\pi$ maps $G$ onto itself, since $\pi$ interchanges $\cP$ and $\cL$, and by Theorem~\ref{thm:symmetry} $G$ has order $2$ or $4$, so $G$ contains an involution $\sigma$ with $\pi\sigma\pi^{-1}=\sigma$.  Then $\pi$ maps the centre $O$ of $\sigma$ to its axis $\ell$, so $a=a'$ and $b=b'$ in Proposition~\ref{prop:involution}.  Over $\CC$ we may take $\pi$ to be the polarity of the conic $x^2+y^2+z^2=0$ and $\sigma=\operatorname{diag}(1,1,-1)$.  In case (a) the polarity pairs each point $q$ of $\cP$ on $\ell$ with the point $\ell\cap\pi(q)$, so the four points form two such pairs or contain the two isotropic points of $\ell$.  In case (b) the value $b=5$ is impossible, since five points on $\ell$ would need ten distinct pairs of lines while $c'=9$.  Four cases remain, namely (a) with two pairs, (a) with the isotropic points, and (b) with $b=1$ or $b=3$.

Since $\pi$ commutes with $\sigma$, the configuration is determined by its points, a point $p$ lying on $\pi(q)$ exactly when $p^{\mathsf T}q=0$, and the graph on $\cP$ recording these incidences is $4$-regular with no two vertices having two common neighbours.  Writing the pairs of points as $(x_k{:}y_k{:}\pm1)$, the quotient of this graph by $\sigma$ is a graph on the pairs with $\ZZ_2$-voltages, and each quotient gives a bilinear system $x_ix_j+y_iy_j\pm1=0$ in $18$ to $24$ unknowns.  We enumerated all quotients in the four cases, with canonical forms computed by \texttt{nauty} and isomorphic partial structures rejected as in Section~\ref{sec:classification}.  This gives $3{,}911{,}536$ structures, and the structures of $(\cP_A,\cL_A)$, $(\cP_B,\cL_B)$ and of the configuration $C$ below, for each of their polarities and involutions, are among them.  Saturating every system by all non-degeneracy conditions in \textsc{Singular} modulo the primes $32003$ and $1{,}000{,}003$, the same $151$ systems, with the same dimensions and degrees, are consistent modulo both, and we solved these exactly over $\QQ$ (over $\QQ(i)$ in the isotropic case); $75$ have real solutions.

For each of these $75$ systems we decomposed the solution set into Galois orbits and checked exactly that no non-degeneracy condition vanishes on any orbit, so every solution is a $(23_4)$ configuration over $\CC$.  The coordinates of a solution in a projective frame of four of its points generate its field of definition; we computed the minimal polynomial of a generator by elimination and counted its real roots exactly, each giving a real configuration.  Grouping these up to projective equivalence yields the following.

\begin{theorem}\label{thm:selfpolar}
\begin{enumerate}
\item[(a)] There are $52$ pairwise projectively inequivalent self-polar geometric $(23_4)$ configurations with a nontrivial collineation group, namely $(\cP_A,\cL_A)$ and $(\cP_B,\cL_B)$, whose group is the Klein four-group, and $50$ configurations whose group has order two.  They have $41$ distinct combinatorial types, and their fields of definition have degrees $1$ to $7$ over $\QQ$.  Four of them, $(\cP_A,\cL_A)$ and three of the new ones, can be defined over $\QQ$.
\item[(b)] Every self-polar geometric $(23_4)$ configuration with a nontrivial collineation group is projectively equivalent to one of these $52$, unless one of the $3{,}911{,}385$ systems that are inconsistent modulo both $32003$ and $1{,}000{,}003$ has a solution over $\overline{\QQ}$.
\end{enumerate}
\end{theorem}

\begin{figure}[t]
\centering
\includegraphics[width=0.74\textwidth,page=3]{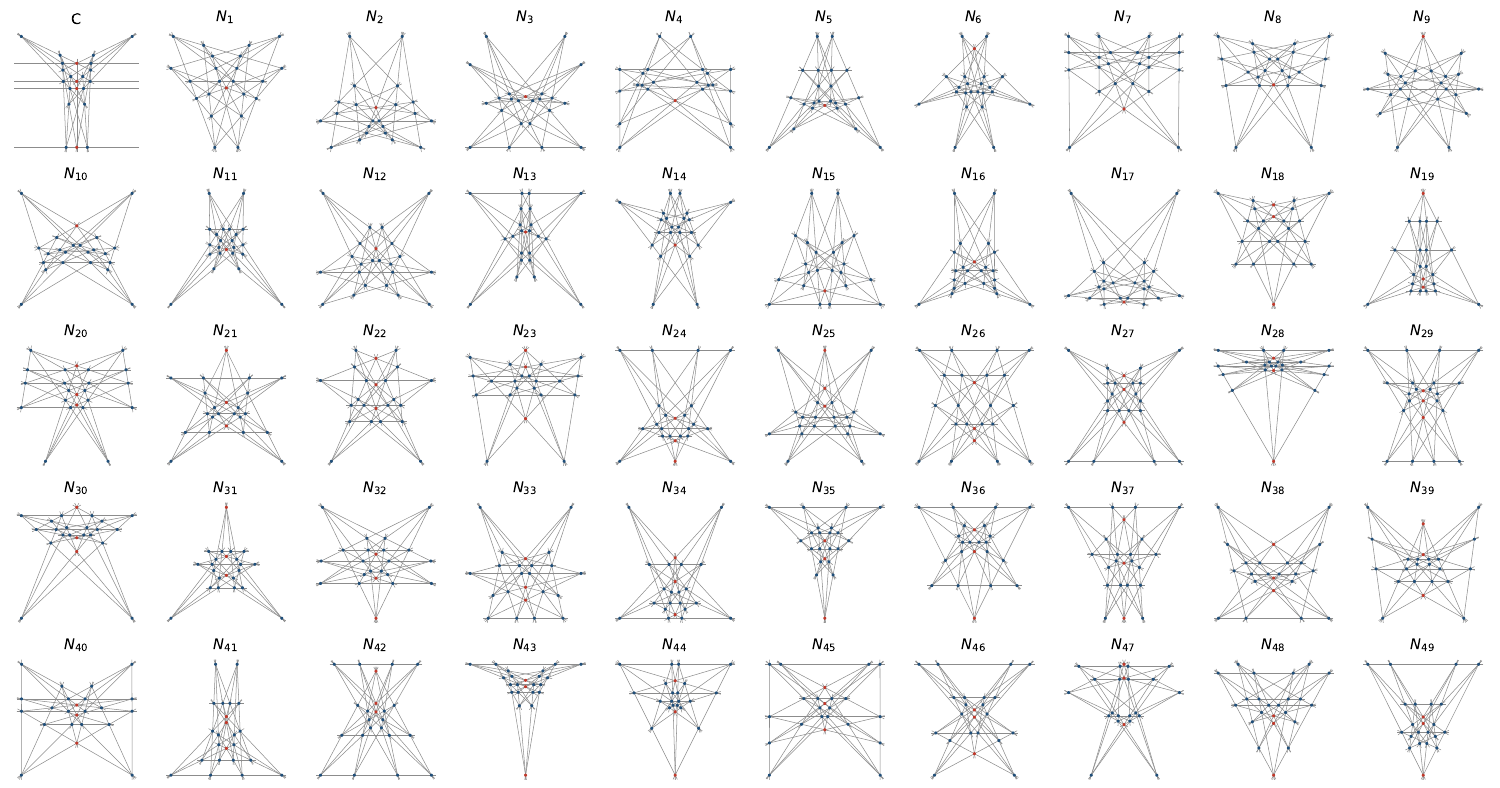}
\caption{The $50$ configurations of Theorem~\ref{thm:selfpolar} whose collineation group has order two, each drawn so that its involution is the reflection in the vertical axis, with fixed points in red. Figure created by the author using Matplotlib 3.10.9 \cite{matplotlib}.}
\label{fig:gallery}
\end{figure}

We write $A$, $B$ for $(\cP_A,\cL_A)$, $(\cP_B,\cL_B)$, $C$ for the first configuration with a single involution that we found, defined over the cubic field $\QQ(\theta)$ with $\theta^3-\theta^2+3\theta-1=0$, and $N_1,\dots,N_{49}$ for the other configurations, in the order in which the enumeration finds them.  Of the $52$ fields of definition, four have degree $1$, those of $A$, $N_{14}$, $N_{31}$ and $N_{38}$, ten have degree $2$, among them that of $B$, eight have degree $3$, among them that of $C$, and $18$, $5$, $6$ and $1$ have degrees $4$, $5$, $6$ and $7$.  Case (a) yields $A$, $B$, $C$, and case (b) with $b=1$ and $b=3$ yields $N_1,\dots,N_{17}$ and $N_{18},\dots,N_{49}$.

\begin{theorem}\label{thm:third}
Let $\cP_N$ consist of the three points $(0{:}1{:}0)$, $(0{:}2{:}{-}3)$, $(0{:}4{:}{-}3)$ and the twenty points $(\pm x{:}y{:}z)$ for $(x{:}y{:}z)$ among
\begin{equation}\label{eq:pointsN}
\begin{gathered}
(1{:}{-}1{:}1),\quad (2{:}{-}1{:}1),\quad (2{:}{-}2{:}1),\quad (2{:}{-}2{:}5),\quad (2{:}{-}7{:}5),\\ (2{:}{-}8{:}5),\quad (4{:}{-}4{:}5),\quad (4{:}{-}6{:}5),\quad (4{:}{-}8{:}5),\quad (8{:}{-}2{:}5),
\end{gathered}
\end{equation}
and let $\cL_N$ consist of their polars with respect to the conic $5x^2+20y^2+40yz+16z^2=0$.  Then $(\cP_N,\cL_N)$ is a geometric $(23_4)$ configuration with integer coordinates, self-polar with respect to a real conic, whose collineation group is $\{1,\sigma\}$ with $\sigma=\operatorname{diag}(-1,1,1)$.  It is the configuration $N_{31}$ of Figure~\ref{fig:gallery}.
\end{theorem}

\begin{proof}
Evaluating the bilinear form $5x_px_q+20y_py_q+20(y_pz_q+z_py_q)+16z_pz_q$ on the $276$ pairs of points shows in integers, as for Theorem~\ref{thm:main}(A), that the $23$ polars are exactly the lines containing four points of $\cP_N$ and that every point lies on four of them.  The Levi graph has $6$ side-preserving automorphisms, so the collineation group, which contains $\sigma$ and has order at most four by Theorem~\ref{thm:symmetry}, is $\{1,\sigma\}$.
\end{proof}

In $23$ of the $52$ configurations some side-preserving automorphisms of the Levi graph, of order $3$, $4$ or $6$ in $14$ of them (among them $B$ and $N_{31}$) and involutions in $9$, are induced by no collineation.

Finally we asked whether these configurations deform.  Fixing four points and one coordinate of every other point, the incidences become $92$ polynomial equations in $38$ unknowns, whose Jacobian matrix has, numerically, full rank at $48$ of the $52$ configurations (smallest singular value at least $10^{-3}$, against rounding errors below $10^{-15}$), so each is isolated in the realization space of its type.  For $N_{27}$, $N_{31}$, $N_{44}$ and $N_{46}$ it has rank $37$, and numerical continuation follows a real curve of realizations through each.  We computed three of these curves exactly in the charts of Section~\ref{sec:classification}, choosing the four points sent to $(1{:}0{:}0)$, $(0{:}1{:}0)$, $(0{:}0{:}1)$, $(1{:}1{:}1)$ so that only three chart parameters $t_0,t_1,t_2$ remain.

\begin{theorem}\label{thm:families}
\textup{(a)} For the types of $N_{27}$, $N_{44}$ and $N_{46}$ there are charts in which the equations, saturated over $\QQ$ by all non-degeneracy conditions, define irreducible nonsingular affine curves of genus $1$, $1$ and $0$ with real points satisfying all these conditions.  Consequently there are infinitely many geometric $(23_4)$ configurations up to projective equivalence.

\textup{(b)} For the type of $N_{27}$ the curve is birational over $\QQ$ to the elliptic curve $y^2+y=x^3-x^2$ of conductor $11$, which has only five rational points.

\textup{(c)} For the type of $N_{46}$ the curve is $t_1+t_2=1$, $(t_0-1)^2t_2^2-(t_0-1)t_2+2t_0=0$, parametrized by $t_0=(1-s^2)/8$ and $t_2=4(s-1)/(s^2+7)$.  For rational $s\notin\{-3,-1,1,3,5\}$ its point is a geometric $(23_4)$ configuration with rational coordinates, self-polar with respect to a real conic, with trivial collineation group for all but finitely many $s$.
\end{theorem}

\begin{proof}
With Singular, in each case the saturated ideal is prime of dimension $1$ with empty singular locus, and the genus was computed by normalization.  Fixing $t_0=1/2$ for $N_{27}$ and $t_0=2$ for $N_{44}$ leaves $2$ and $4$ points, all real by Sturm sequences, and adding any non-degeneracy condition gives the unit ideal; for $N_{46}$ see (c).  Near such a nonsingular real point the real points form an arc (implicit function theorem), on which the open non-degeneracy conditions hold.  A chart fixes four points, so a projective equivalence class gives at most as many chart points as the Levi graph has automorphisms, and the arc contains infinitely many classes.

For (b), the curve projects birationally, as $t_2=(t_1-t_0)/(t_0^2-2t_0+t_1)$ on it, onto $t_1^2+t_0(t_0-2)(t_0+1)t_1=t_0(2t_0^2-4t_0+1)$, which $x=1/t_0$, $y=(t_1-t_0)/(t_0^2(t_0-1))$ map birationally onto $y^2+y=x^3-x^2$, the modular curve $X_1(11)$.

For (c), all points and lines become polynomial vectors in $s$, every incidence holds identically, and the non-degeneracy conditions fail at no other rational $s$.  Solving $Bp\parallel\delta(p)$ over $\QQ(s)$ for a duality $\delta$ of the Levi graph gives a symmetric matrix $B(s)$ whose determinant $-256(s-5)^2(s-3)(s-1)^4(s+3)^2(s^2+7)$ vanishes only at excluded values, and which is indefinite.  A side-preserving automorphism of the Levi graph that is induced by a collineation for infinitely many $s$ is induced for every $s$, and for $s=-1/3$ none of the three nontrivial ones is.  This member consists, after scaling, of the points $(1{:}0{:}0)$, $(0{:}1{:}0)$, $(0{:}0{:}1)$, $(1{:}1{:}1)$, $(1{:}1{:}0)$, $(4{:}7{:}0)$, $(0{:}3{:}8)$, $(0{:}3{:}{-}1)$, $(3{:}3{:}{-}1)$, $(3{:}6{:}2)$, $(3{:}6{:}7)$, $(9{:}0{:}{-}7)$, $(9{:}3{:}1)$, $(9{:}3{:}11)$, $(9{:}6{:}1)$, $(9{:}12{:}5)$, $(9{:}18{:}1)$, $(9{:}21{:}{-}7)$, $(9{:}30{:}{-}7)$, $(15{:}12{:}7)$, $(18{:}6{:}7)$, $(33{:}21{:}7)$, $(81{:}42{:}49)$ and their polars with respect to $7x^2+10y^2+9z^2-14xy-24xz+6yz=0$, checked in integers as for Theorem~\ref{thm:third}.
\end{proof}

\section{Discussion and future work}\label{sec:future}

The results settle the existence question for geometric $(n_4)$ configurations and explain why $n=23$ was the hardest case.  Theorem~\ref{thm:symmetry} shows that symmetry, the tool behind almost every earlier construction, is limited to a Klein four-group at $n=23$, and Theorem~\ref{thm:classification} shows that this largest symmetry is attained by exactly two configurations.  Beyond them the picture changes completely.  Theorems~\ref{thm:selfpolar} and~\ref{thm:families} show that $(23_4)$ configurations are plentiful once less symmetry is required, with infinitely many of them having rational coordinates, so the difficulty at $n=23$ lay in finding the first one rather than in a scarcity of configurations.  Methodologically, imposing the largest possible symmetry, searching over finite fields and certifying the answer exactly succeeded where the earlier symmetric constructions could not.

Theorem~\ref{thm:selfpolar}(b) is conditional only on the $3{,}911{,}385$ systems that are inconsistent modulo both primes.  A system with a solution over $\overline{\QQ}$ remains consistent modulo all but finitely many primes, so a missed configuration would need both $32003$ and $1{,}000{,}003$ to be exceptional for the same system.  Certifying these systems over $\QQ$, as we did for the $4018$ systems of Section~\ref{sec:classification}, would remove the condition.

The next step is the configurations with an involution that are not self-polar.  They exist, since exact members of the families of $N_{31}$ and $N_{44}$ have an involution but no polarity, and Proposition~\ref{prop:involution} reduces them to nine orbit patterns.  Without a polarity the lines are independent of the points, which roughly doubles the number of unknowns, the patterns with $b=5$ or $b'=5$ are no longer excluded, and the families of $N_{31}$ and $N_{44}$ show that the solution sets need not be finite, so the systems must be decomposed into components rather than only solved.

\begin{question}\label{q:spaces}
Which combinatorial $(23_4)$ types are realizable, and which curves occur as their realization spaces?  Is any of dimension two or more?
\end{question}

Of the $52$ configurations of Theorem~\ref{thm:selfpolar}, $48$ are isolated, and Theorem~\ref{thm:families} gives two elliptic curves and a rational curve, one birational to the modular curve $X_1(11)$; the curve through $N_{31}$ has not yet been computed exactly.  By Mn\"ev's universality theorem \cite{Mne88}, realization spaces of point-line incidence structures are, up to stable equivalence, arbitrary semialgebraic sets, so it is natural to ask how much of this already occurs for $(n_4)$ configurations, and why a modular curve appears at $n=23$.

\begin{question}\label{q:five}
Can the finite-field search of Section~\ref{sec:search} find geometric $(n_5)$ configurations for small $n$?
\end{question}

They exist for all $n\ge576$ \cite{BGP21}, but the least such $n$ is unknown.  By Proposition~\ref{prop:five}, a search at a prime $n\not\equiv1\pmod5$ can impose at most a Klein four-group, as at $n=23$.  The next step would be an analogue of Proposition~\ref{prop:structure} for Klein-symmetric $(n_5)$ configurations, which would reduce such a search to quotient structures as in Section~\ref{sec:classification}.

Finally, the classification of Section~\ref{sec:classification} can be run at the next primes $n\equiv3\pmod4$, namely $31$ and $43$, where by Theorem~\ref{thm:symmetry} the Klein four-group is again the largest possible symmetry, and would determine the most symmetric $(31_4)$ and $(43_4)$ configurations.

\end{document}